\documentclass[a4paper,11pt]{amsart} 
\usepackage{amssymb, amsmath}
\usepackage{amscd}
\numberwithin{equation}{section}
\usepackage{epsfig}
\usepackage{amsmath}
\usepackage{amsfonts,amssymb,amsopn}

\usepackage{amsthm}
\usepackage{verbatim}
\usepackage{color}
\usepackage[color,all]{xy}

\newcommand{\PP}{\mathbb{P}}
\newcommand{\K}{\mathbb{K}}
\newcommand{\bT}{\mathbb{T}}

\newcommand{\cE}{{\mathcal E}}
\newcommand{\cG}{{\mathcal G}}
\newcommand{\cO}{{\mathcal O}}
\newcommand{\cM}{{\mathcal M}}
\newcommand{\cP}{{\mathcal P}}
\newcommand{\cT}{{\mathcal T}}
\newcommand{\cV}{{\mathcal V}}

\newcommand{\tB}{\tilde{B}}
\newcommand{\tU}{\tilde{U}}
\newcommand{\tV}{\tilde{V}}
\newcommand{\tphi}{\widetilde{\phi}}

\newcommand{\hG}{\hat{G}}

\newcommand{\Coker}{\mathrm{Coker}}
\newcommand{\End}{\mathrm{End}\,}
\newcommand{\Ker}{\mathrm{Ker}}
\newcommand{\tr}{\mathrm{tr}}
\newcommand{\Image}{\mathrm{Im}\,}
\newcommand{\Iden}{\mathrm{Id}}
\newcommand{\Hom}{\mathrm{Hom}}
\newcommand{\rank}{\mathrm{rk}\,}

\newcommand{\Pic}{\mathrm{Pic}}

\newcommand{\ev}{\mathrm{ev}}
\newcommand{\isom}{\xrightarrow{\sim}}

\newcommand{\Kc}{K_{C}}
\newcommand{\Oc}{{{\cO}_{C}}}

\newcommand{\Gr}{\mathrm{Gr}}
\newcommand{\Spec}{\mathrm{Spec}\,}
\newcommand{\Sec}{\mathrm{Sec}}
\newcommand{\Quot}{\mathrm{Quot}}

\newcommand{\Sing}{\mathrm{Sing}}

\newcommand{\mult}{\mathrm{mult}}

\newcommand{\bv}{B^k_{n,e} (V)}
\newcommand{\Bkg}{\left( B^k_{r,d} \right)_{\mathrm{PTI}}}
\newcommand{\gkv}{G^k_{n, e} (V)}
\newcommand{\tUrd}{\widetilde{U(r,d)}}
\newcommand{\tUne}{\widetilde{U(n,e)}}

\newcommand{\PTI}{\mathrm{PTI}}

\newtheorem{theorem}{{\textbf Theorem}}[section]
\newtheorem{proposition}[theorem]{{\textbf Proposition}}
\newtheorem{corollary}[theorem]{{\textbf Corollary}}
\newtheorem{lemma}[theorem]{{\textbf Lemma}}
\newtheorem{defn}[theorem]{{\textbf Definition}}

\newtheorem{remit}[theorem]{{\textbf Remark}}

\newenvironment{remark}{\begin{remit}\rm}{\end{remit}}
\newenvironment{definition}{\begin{defn}\rm}{\end{defn}}

\title{Nonemptiness and smoothness of twisted Brill--Noether loci}

\author{George H. Hitching, Michael Hoff and Peter E. Newstead}

\address{G. H. Hitching, Oslo Metropolitan University, Postboks 4, St. Olavs plass, 0130 Oslo, Norway.}
\email{gehahi@oslomet.no}
\address{M. Hoff, Universit\"at des Saarlandes, Campus E2 4, D-66123 Saarbr\"ucken, Germany.}
\email{hahn@math.uni-sb.de}
\address{P. E. Newstead, Department of Mathematical Sciences, University of Liverpool, Peach Street, Liverpool, L69 7ZL, UK.}
\email{newstead@liv.ac.uk}
\date{\today}
\keywords{Brill-Noether loci, Petri trace map}
\subjclass[2010]{14H60 (14H51)}

\begin{document}

\begin{abstract} Let $V$ be a vector bundle over a smooth curve $C$. In this paper, we study twisted Brill--Noether loci parametrising stable bundles $E$ of rank $n$ and degree $e$ with the property that $h^0 (C, V \otimes E) \ge k$. We prove that, under conditions similar to those of Teixidor i Bigas and of Mercat, the Brill-Noether loci are nonempty, and in many cases have a component which is generically smooth and of the expected dimension. Along the way, we prove the irreducibility of certain components of both twisted and ``nontwisted'' Brill--Noether loci. We describe the tangent cones to the twisted Brill-Noether loci. We end with an example of a general bundle over a general curve having positive-dimensional twisted Brill--Noether loci with negative expected dimension.
\end{abstract}

\maketitle
\section{Introduction}\label{intro}
Let $C$ be a smooth projective curve over an algebraically closed field $\K$ of characteristic zero. A fundamental feature of the geometry of $C$ and $\Pic^d (C)$ is the Brill--Noether locus
\begin{equation} W_d^r (C) \ = \ \{ L \in \Pic^d(C) : h^0 (C, L) \ge r+1 \} . \label{ClassicalBN} \end{equation}
These objects have been much studied. The expected dimension of $W_d^r (C)$ is the Brill--Noether number $\rho(g,d,r)=g-(r+1)(g-d+r)$ where $g$ is the genus of $C$; every irreducible component has dimension $\ge\rho(g,d,r)$ and a great deal is known about these loci (for details, see Section \ref{back}).

A natural generalisation of (\ref{ClassicalBN}) to vector bundles of higher rank is given as follows. We denote by $U(n, e)$ the moduli space of stable bundles of rank $n$ and degree $e$ over $C$. This is an irreducible quasiprojective variety of dimension $n^2(g-1) + 1$. The \textsl{generalised Brill--Noether locus} $B^k_{n,e}$ is defined set-theoretically by
\[ B^k_{n,e} \ = \ \{ E \in U(n,e) : h^0 (C, E) \ge k \} . \]
(In this notation, $W_d^r (C)$ is written $B_{1, d}^{r+1}$.) These loci have also been studied in much detail, although the results for the case $n=1$ do not necessarily generalise. Brill--Noether loci are also closely related to moduli of \textsl{coherent systems}, that is, pairs $(V, \Lambda)$ where $V$ is a vector bundle and $\Lambda$ a subspace of $H^0 (C, V)$ of a fixed dimension.

In the present work we study another generalisation of $B^k_{n,e}$, which to our knowledge was first defined in \cite[\S 2]{TiBT}. Fix a vector bundle $V$ over $C$ of rank $r$ and degree $d$ (not necessarily semistable). Then the \textsl{twisted Brill--Noether locus} $\bv$ is defined set-theoretically by
\[ \bv \ := \ \{ E \in U(n, e) : h^0 ( C, V \otimes E ) \ge k \} . \]

As outlined in \cite[\S 1]{TiB14}, the construction of $B^k_{n,e}$ in \cite[\S 2]{GTiB} is easily generalised to $\bv$, substituting a vector bundle $V$ for $\Oc$ in the appropriate places. (In \S 2 we will give a slightly more general version of this construction.) In particular, $\bv$ is a determinantal locus. Thus if $h^0 (C, V \otimes E) = k$, then the expected dimension of $\bv$ at $E$ is given by the \textsl{twisted Brill--Noether number}
\begin{align*} \rho^k_{n,e}(V) \ :&= \ \dim U(n, e) - k \left( k - \chi (C, V \otimes E) \right) \\
 &= \ n^2 ( g-1) + 1 - k \left( k - re-nd + rn(g-1) \right) . \end{align*}
Provided that this number is less than $\dim U(n,e)$, every irreducible component of $\bv$ has dimension at least equal to $\rho^k_{n,e}(V)$. If $\rho^k_{n,e}(V)\ge\dim U(n,e)$ then $\bv=U(n,e)$.

In the case $k=1$ with $V$ of integral slope $h$, we have $\rho^1_{1,g-1-h}=g-1$ and $B^1_{1,g-1-h}(V)$ is expected to be a divisor $\Theta_V$ in the Picard variety $\Pic^{g-1-h}(C)$. When $\Theta_V$ is a divisor, it is called a \textit{generalised theta divisor}. These have been much studied; see \cite{B06} for an overview. See also \cite{Br} for results on the singular loci of $\Theta_V$. It can also happen for special $V$ that $B^1_{1,h}(V)$ fails to be a divisor; see \cite{Ray}, \cite{Po} and \cite{Pa} for some examples. If $V$ does not have integral slope, then the theta divisor of $V$, if it exists, belongs to $U(n,e)$ for some $n \ge 2$. See \cite{Po13} for a survey of results on this type of generalised theta divisor. 

Note also the connection with varieties of subbundles of a vector bundle $V$. If we denote by $M_{n,e}(V)$ the variety of stable subbundles of $V$ of rank $n$ and degree $e$, there is a natural morphism $M_{n,e}(V)\to B^1_{n,e}(V)$. In particular, when $n=1$ and $e$ is maximal, this is a question of maximal line subbundles. In the case $r=2$, these have been studied for a long time, dating back to \cite{LN}; for more recent work and all $r$, see \cite{Oxb}. For $n > 1$, see \cite{Hol}, \cite{LanN}, \cite[Theorem 0.3]{RTiB} and \cite{BPL}.

When $n=1$, it turns out that the basic results of classical Brill-Noether theory generalise, at least when $V$ is a general stable bundle; for details, see Theorem \ref{LTiB}. This study was initiated in \cite{Hir2}. Our purpose in this article is to study the case $n>1$.

In \S\ref{back}, we give more details on some of the background material mentioned in the introduction. In \S \ref{preliminaries}, we construct the twisted Brill--Noether locus $B^k ( \cV, \cE )$ associated to a pair of families of bundles over $C$, with $\bv$ as a special case. After listing some elementary properties, we develop some more tools. We construct parameter spaces for certain ``twisted coherent systems'', generalising the loci $G^r_d(C)$ in \cite{ACGH} and the moduli spaces of $\alpha$-stable coherent systems, although we do not discuss stability or moduli.
Furthermore, in \S \ref{semistable}, we discuss the twisted Brill--Noether loci $\tB^k_{n, e}(V)$ where strictly semistable bundles are admitted.

In \S \ref{irred} we give two applications of the machinery set up in \S \ref{preliminaries}. In Theorem \ref{IrrPGfamilies}, we generalise Theorem \ref{LTiB}(5) to families of vector bundles which are general in the sense of \cite{TiB14}. We also find that, for a certain range of values of $k$, the Brill--Noether locus $B^k_{r,d}$ possesses a uniquely determined irreducible component $\Bkg$ (Theorem \ref{PTIIrrComp}); this is interesting because very little is known in general about irreducibility of $B^k_{r,d}$ for $k\ge2$ and $r>1$. 

In \S \ref{chapternonemptiness}, we turn to twisted Brill--Noether loci $\bv$ for $n > 1$ and $k \ge 2$, which to our knowledge remain relatively little studied. We will answer some of the basic questions on nonemptiness and smoothness in this case. Our first result is:

\begin{theorem} \label{nonemptiness} Let $C$ be a smooth curve of genus $g$ and $V$ any vector bundle of rank $r$ and degree $d$ over $C$. Let $e_0$ and $k_0$ be integers satisfying $\rho^{k_0}_{1, e_0}(V) \ge 1$. Then for all $n \ge 2$, for all $e \ge ne_0 + 1$ (resp., $e \ge ne_0$) and for $1 \le k \le nk_0$, the twisted Brill--Noether locus $B^{nk_0}_{n,e}(V)$ (resp., $\tB^{nk_0}_{n,e} (V)$) is nonempty. \end{theorem}

\noindent This directly generalises both the main result and the construction of \cite{Mer}.

We are also interested in generically smooth components of the loci $B^k_{n,e}(V)$. Our approach turns out to require the existence of certain bundles with well-behaved rank-$1$ twisted Brill--Noether loci and which are generically generated. In \S \ref{genPTI} we construct such bundles for some values of $r$, $g$ and $d$. We then prove in \S \ref{smoothness} our main result:

\begin{theorem} \label{main} Let $C$ be a general curve and $r$, $l$,  $m$ integers with $l := \left\lfloor \frac{g}{r} \right\rfloor$ and $0 \le m \le l-1$.  If $m=0$, suppose that $g \not\equiv 0 \mod r$. Write $k_0 = l - m$ and let $d$, $e_0$ be integers with $d + re_0 = r(g-2) + k_0$. Suppose that $e$ and $k$ are integers satisfying
\begin{equation}\label{eq12}
ne_0 + 1 \ \le \ e \ \le \ n(e_0+1) \quad \quad \hbox{and} \quad \quad re + nd - rn(g-1) \ \le \ k \ \le \ nk_0 . 
\end{equation}
Then, for general $V \in U(r, d)$, the twisted Brill--Noether locus $B^k_{n,e}(V)$ has a component $B^k_{n,e}(V)_0$ which is generically smooth and of the expected dimension.
\end{theorem}

The conditions here may look rather restrictive. Note however that \eqref{eq12} is more or less equivalent to those of \cite{TiB91} and \cite{Mer}. Moreover, if $k \le re+nd-rn(g-1)$, then $\bv=U(n,e)$.

To prove that $\bv$ is generically smooth and of the expected dimension at a point $E$, we have to show that the generalised trace map
\[ H^0 ( C, V \otimes E) \otimes H^0 (C, \Kc \otimes E^* \otimes V^*) \to H^0 (C, \Kc \otimes \End E) \]
is injective (details in \S \ref{preliminaries}). For this type of question, Teixidor i Bigas's generalisation of limit linear series to vector bundles of higher rank has been applied in many situations; for example \cite{TiB91}, \cite{TiB07}, \cite{CMTiB} and \cite{TiB14}. Although we do not use limit linear series directly, several of our proofs rely on the main result of \cite{TiB14}.

In Section \ref{TangentCones}, we consider the tangent cones of $\bv$, which can be studied using the same techniques as in \cite{ACGH} and \cite{CTiB}. Using Theorem \ref{main}, we describe the tangent cones as determinantal varieties and compute their degrees. We also give a geometric description of the tangent cones for large values of $h^0 ( V \otimes E )$, generalising \cite[VI, Theorem 1.6 (i)]{ACGH} on secant varieties of canonical curves.

Finally, in Section \ref{negative}, we describe some twisted Brill-Noether loci which are non-empty but have negative expected dimension. These are closely connected with varieties of maximal subbundles and we exploit results of \cite{Oxb} in discussing them. This gives another motivation for studying twisted Brill--Noether loci: As these examples arise for general $C$ and $V$ (with prescribed numerical properties), twisted Brill--Noether loci give a way of systematically obtaining determinantal varieties of larger than expected dimension. This line of research will be further pursued in the future.

\subsection*{Acknowledgements:} We thank Andr\'e Hirschowitz for answering several questions on rank stratifications.

\subsection*{Notation:} We work over an algebraically closed field $\K$ of characteristic zero. 
We denote a locally free sheaf and the corresponding vector bundle by the same letter. If $F$ is an $\Oc$-module, we abbreviate $H^i(C, F)$, $h^i(C, F)$ and $\chi(C, F)$ respectively to $H^i(F)$, $h^i(F)$ and $\chi(F)$. If $D$ is a divisor on $C$, we denote $F \otimes \Oc(D)$ by $F(D)$. The fibre of a bundle $V$ at $p \in C$ will be denoted $V|_p$. If $\cV \to S \times C$ is a family of bundles parametrised by $S$, we denote the restriction $\cV|_{\{s\} \times C}$ by $\cV_s$. We suppose throughout that $k\ge1$.

\section{Background}\label{back}

In this section, we expand on the background to our paper already referred to in the introduction.The fundamental results on $W_d^r (C)$ are as follows (see \cite[Chapter V]{ACGH}):

\begin{enumerate}
 \item [(i)] \textsl{Existence theorem:} For any curve, $W^r_d(C)$ is non-empty if $\rho(g,d,r)\ge 0$. 
 \item [(ii)] \textsl{Connectedness theorem:} For any curve, $W^r_d(C)$ is connected if $\rho(g,d,r)\ge 1$. 
 \item [(iii)] \textsl{Dimension theorem:} For a general curve, $W^r_d(C)=\emptyset$ if $\rho(g,d,r)< 0$; if $0\le\rho(g,d,r)\le g$, then $W^r_d(C)$ has pure dimension $\rho(g,d,r)$. 
 \item [(iv)] \textsl{Smoothness theorem:} For a general curve, $\Sing(W^r_d(C))=W^{r+1}_d(C)$. 
 \item [(v)] For a general curve, $W_d^r (C)$ is irreducible if $\rho(g, d, r) \ge 1$. 
\end{enumerate}
\noindent Note that if $\rho(g,d,r)\ge g$, then $W^r_d(C)=\Pic^d(C)$.

 Many of the basic questions on nonemptiness of $B^k_{n,e}$ were answered in \cite{TiB91} and \cite{Mer}, and more detailed results have been obtained in several cases. However, analogues of statements (i)--(v) above may be false in higher rank. See \cite{GTiB} for an overview of the theory and a survey of results and techniques. For the links between Brill-Noether theory and the moduli of coherent systems, see \cite{Brad} and \cite{BGPMN}. See also \cite{New11} for a survey of results and open problems on coherent systems; note however that there are many more recent results in this area.

When $B^k_{n, e}(V)$ has the expected dimension, one has
\[ B^{k+1}_{n, e}(V) \ \subseteq \ \Sing \left( B^k_{n, e}(V) \right) . \]
This containment may, however, be strict. See \cite{CTiB} for a detailed discussion of singular points $E \in \Theta_V$ satisfying $\mult_E \left( \Theta_V \right) > h^0 (C, V \otimes E)$. 

As already remarked, for $n=1$, analogues of several of the fundamental results for $W_d^r (C)$ are also valid for sufficiently general bundles of higher rank:

\begin{theorem} \label{LTiB} Let $C$ be any curve of genus $g \ge 2$. Let $V$ be a vector bundle of rank $r$ and degree $d$. Let $k \ge 1$ and $e$ be integers.
\begin{enumerate}
\renewcommand{\labelenumi}{(\arabic{enumi})}
\item If $\rho^k_{1, e}(V) \ge 0$, then $B^k_{1, e}(V)$ is non-empty.
\item If $\rho^k_{1, e}(V) \ge 1$, then $B^k_{1, e}(V)$ is connected.
\item Suppose $C$ is a general curve and $V$ a general bundle. If $\rho^k_{1, e}(V) < 0$, then $B^k_{1, e}(V)$ is empty. If $0\le\rho^k_{1,e}(V) \le g$, then $B^k_{1, e}(V)$ has pure dimension $\rho^k_{1,e}(V)$.
\item Suppose $C$ is a general curve and $V$ a general bundle. Then $\Sing \left( B^k_{1, e}(V) \right) = B^{k+1}_{1, e}(V)$.
\item Suppose $C$ is a Petri curve and $V$ a general bundle. If $\rho^k_{1, e}(V) \ge 1$, then $B^k_{1,e}(V)$ is irreducible.
\end{enumerate} \end{theorem}

\begin{proof} Statement (1) was proven in \cite{Gh} for general $V$, and for all $V$ in \cite[(2.6)]{Laz}. Part (2) is \cite[(2.7)]{Laz}. Parts (3) and (4) follow from \cite{TiB14}. Lastly, (5) is \cite[Th\'eor\`eme 1.2]{Hir2}. \end{proof}

The infinitesimal study of the Brill-Noether loci is the key to the proofs of (3) and (4). We do not include this here because we will describe it in detail for $B^k_{n,e}(V)$ (and indeed for families of bundles) in the next section.

\section{Preliminaries on twisted Brill--Noether loci} \label{preliminaries}

Although generalised Brill--Noether loci are by now very familiar objects, there are fewer sources focusing primarily on twisted Brill--Noether loci. We will therefore give a detailed introduction to the subject with emphasis on functorial aspects.

\subsection{The twisted Brill--Noether locus of a pair of families} \label{BNfamily}

Let $\cV \to S \times C$ be a family of bundles of rank $r$ and degree $d$, and $\cE \to T \times C$ a family of bundles of rank $n$ and degree $e$. Set-theoretically, we define
\[ B^k ( \cV, \cE ) \ := \ \{ (s, t ) \in S \times T : h^0 ( \cV_s \otimes \cE_t ) \ge k \} . \]
Scheme-theoretically, this is a determinantal locus, as we will now show using a standard construction. Let $D$ be an effective divisor of large degree on $C$ satisfying $h^1 ( \cV_s \otimes \cE_t ( D)) = 0$ for all $(s, t) \in S \times T$. We have a diagram of projections
\[ \xymatrix{ & S \times T \times C \ar[dl]_{p_{13}} \ar[d]^{p_{12}} \ar[dr]^{p_{23}} \ar[r]^-{p_3} & C \\
 S \times C & S \times T & T \times C . } \]
Then over $S \times T \times C$, we have the short exact sequence
\[ 0 \to p_{13}^*\cV \otimes p_{23}^*\cE \to p_{13}^*\cV \otimes p_{23}^* \cE \otimes p_3^* \Oc ( D ) \to \frac{p_{13}^*\cV \otimes p_{23}^* \cE \otimes p_3^* \Oc ( D ) }{p_{13}^*\cV \otimes p_{23}^*\cE} \to 0 . \]
Pushing down to $S \times T$, we obtain a complex $\gamma \colon K^0 \to K^1$ of locally free sheaves satisfying
\[ \Ker \left( \gamma_{(s, t)} \right) \ \cong \ H^0 ( \cV_s \otimes \cE_t ) \quad \hbox{and} \quad \Coker \left( \gamma_{(s, t)} \right) \ \cong \ H^1 ( \cV_s \otimes \cE_t ) \]
for each $(s, t) \in S \times T$. Then $B^k (\cV, \cE )$ is the locus defined by the $(\rank K^0 - k) \times (\rank K^0 - k)$-minors of $\gamma$. In particular (see \cite[Chapter 2]{ACGH}), the locus $B^k (\cV, \cE)$ has a natural scheme structure and every component of it has dimension at least
\begin{equation} \dim S + \dim T - k(k -  re-nd  + rn(g-1)) . \label{expdimBk} \end{equation}
From the determinantal description it also follows that $B^{k+1} ( \cV, \cE) \subseteq \Sing \left( B^k ( \cV, \cE) \right)$, and moreover that the loci $B^k ( \cV, \cE )$ define a rank stratification on $S \times T$. We will return to this aspect in \S \ref{stratification}.

\begin{remark} \label{functoriality} This construction is symmetric in $\cV$ and $\cE$. It is functorial in the sense that if $\phi \colon S' \to S$ and $\psi \colon T' \to T$ are morphisms, then
\[ B^k \left( (\phi \times \Iden_C) ^* \cV , (\psi \times \Iden_C)^* \cE \right) \]
is defined by the $(\rank K^0 - k) \times (\rank K^{0} - k)$-minors of $(\phi \times \psi)^* \gamma$. \end{remark}

\begin{definition} Let $V$ be a vector bundle of rank $r$ and degree $d$, considered as a family over $\Spec \K \times C$. By \cite[Proposition 2.4]{NR75} there exists an \'etale cover $\tUne \to U(n, e)$ (which can be taken to be the identity if $\gcd(n, e) = 1$) which carries a Poincar\'e family $\cE \to \tUne \times C$. Then the twisted Brill--Noether locus $\bv$ is defined as the image of the moduli map $B^k ( V, \cE) \to U(n, e)$. Writing $\chi = re + nd - rn(g-1)$, the expected dimension of $\bv$ is the Brill--Noether number
\[ \rho^k_{n,e} (V) \ = \ \rho^k_{n, e, r, d} \ := \ \dim U(n, e) - k ( k - \chi ) . \]
\end{definition}

\noindent The following is straightforward to check:

\begin{proposition}\label{prop23} Let $V$ be any bundle of rank $r$ and degree $d$ over $C$.
\begin{enumerate}
\item[(1)] $\bv$ is a proper sublocus of $U(n, e)$ only if $k > \chi$.
\item[(2)] If $V$ is stable, then $\bv$ is non-empty only if $re + nd > 0$ or $(n,e)=(r,-d)$. In the latter case, $B^1_{r, -d}(V) = \{ V^* \}$ and $B^k_{r,-d}$ is empty for $k \ge 2$.
\item[(3)] If $V$ is semistable, then $\bv$ is non-empty only if $re + nd > 0$ or $re + nd = 0$ and $r \ge n$.
\item[(4)] For any line bundle $L$ of degree $\ell$, there is a canonical isomorphism
\[ \xymatrix{ \bv \ar[r]^-{\sim} & B^k_{n, \, e-n\ell}(V \otimes L) } \]
given by $E \mapsto L^{-1} \otimes E$.
\item[(5)] Via Serre duality, the association $E \mapsto E^*$ gives an isomorphism
\[ \bv \ \xrightarrow{\sim} \ B^{k - \chi}_{n, -e} (\Kc \otimes V^*) . \]
\end{enumerate} \end{proposition}


\subsection{The tangent spaces of $B^k ( \cV, \cE)$}

We now recall some standard facts on deformations of bundles and sections. Suppose $W$ is a vector bundle with $h^0 (W) \ge 1$. Let $v \in H^1 (\End W)$ be a first-order infinitesimal deformation of $W$. By the argument in \cite[\S 2]{GTiB}, a section $s \in H^0 (W)$ is preserved by $v$ if and only if $s \cup v = 0$ in $H^1 (W)$. Thus the space of deformations preserving all sections of $W$ is exactly
\[ \Ker \left( \cup \colon H^1 (\End W) \to \Hom \left( H^0 ( W ), H^1 (W) \right) \right) . \]
We are interested in the case where $W$ is of the form $V \otimes E$ and $v$ is the class of a product of deformations $b \in H^1 ( \End V )$ and $h \in  H^1 ( \End E )$ of $V$ and $E$ respectively. By for example inspecting \v{C}ech cocycles, we see that $v = c(b, h)$ , where
\begin{equation} c(b, h) \ := \ b \otimes \Iden_E + \Iden_V \otimes h \ \in \ H^1 ( \End (V \otimes E)) . \label{Mapc} \end{equation}
\
More generally, let us consider again the families $\cV \to S \times C$ and $\cE \to T \times C$. Suppose $(s, t) \in S \times T$ is such that $h^0 ( \cV_s \otimes \cE_t ) = k$. We have a composed map
\begin{multline} T_s S \oplus T_t T \xrightarrow{\kappa} H^1 ( \End \cV_s ) \oplus H^1 ( \End \cE_t ) \xrightarrow{c} \\ H^1 \left( \End \left( \cV_s \otimes \cE_t \right) \right) \xrightarrow{\cup} \Hom \left( H^0 ( \cV_s \otimes \cE_t ), H^1 ( \cV_s \otimes \cE_t ) \right) \label{TangSpSeq} \end{multline}
where $\kappa$ is the Kodaira--Spencer map.

\begin{proposition} Suppose that $h^0 ( \cV_s \otimes \cE_t) = k$.
\begin{enumerate}
\item[(1)] The Zariski tangent space to $B^k ( \cV, \cE)$ at $(s, t)$ is given by
 \begin{equation} T_{(s, t)} B^k ( \cV, \cE ) \ = \ \Ker ( \cup \circ c \circ \kappa ) . \label{ZariskiTangentSpace} \end{equation}
\item[(2)] In particular, suppose that $S \times T$ is smooth at $(s, t)$. Then $B^k ( \cV, \cE )$ is smooth and of the expected dimension (\ref{expdimBk}) at $(s, t)$ if and only if $\cup \circ c \circ \kappa$ is surjective. \end{enumerate}\label{BkSmoothV1} \end{proposition}

\begin{proof} Since $T_{(s, t)} B^k ( \cV, \cE )$ consists of those deformations preserving all sections of $\cV_s \otimes \cE_t$, we obtain (1). By \eqref{expdimBk} and (1), we see that $B^k(\cV,\cE)$ is smooth of the expected dimension at $(s,t)$ if and only if 
\[ \dim S + \dim T - \dim\Ker ( \cup \circ c \circ \kappa )=k(k-re-nd+rn(g-1)).\]
By (\ref{TangSpSeq}), this is equivalent to the surjectivity of $\cup \circ c \circ \kappa$.
\end{proof}

\subsection{The Petri trace map}

For any vector bundle $W$, it is well known that via Serre duality, $\cup \colon H^1 ( \End (W) ) \to \Hom \left( H^0 (W), H^1 (W) \right)$ is dual to the \textsl{Petri multiplication map} 
\begin{equation*}\mu \colon H^0 ( W ) \otimes H^0 ( \Kc \otimes W^*) \to H^0 (\Kc \otimes \End W).
\end{equation*} 
Let us use this map to reformulate (\ref{ZariskiTangentSpace}).

Firstly, some notation: For bundles $V$ and $E$, there is a vector bundle map
\[ c \colon (\End V) \oplus ( \End E ) \ \to \End ( V \otimes E ) \]
inducing the cohomology map (\ref{Mapc}) considered above. We write $c_V$ and $c_E$ for the restrictions to the first and second factors respectively. Recall also that for any bundle $W$, the transpose gives a canonical identification of $\End W$ and $\End W^*$, which we will use freely.

Fix a vector bundle $V$. If we identify $\End V$ with $(\End V)^*$ by the trace pairing, a diagram chase shows that the trace map $\tr \colon \End V \to \Oc$ is dual to the map $\Oc \to \End V$ given by $\lambda \mapsto \lambda \cdot \Iden_V$. 
 Thus for any bundle $E$, tensoring $\tr \colon \End V \to \Oc$ by $\End E$, we obtain a linear map
\[ \End ( V \otimes E ) \ \cong \ \End V \otimes \End E \to \End E \]
which is dual to $c_E$, and an induced map
\[ \tr_E \colon H^0 ( \Kc \otimes \End (E \otimes V) ) \ \to \ H^0 ( \Kc \otimes \End E) . \]
By Serre duality and the above discussion, $\tr_E$ is dual to
\[ c_E \colon H^1 ( \End E ) \to H^1 ( \End (V \otimes E)) . \]
Then by linear algebra, $c \colon H^1 ( \End V ) \oplus H^1 ( \End E ) \ \to \ H^1 ( \End ( V \otimes E) )$ is dual to
\[ ( \tr_V , \tr_E ) \colon H^0 ( \Kc \otimes \End (V \otimes E)) \ \to \ H^0 ( \Kc \otimes \End V) \oplus H^0 ( \Kc \otimes \End E) . \]

We can now formulate a dual version of Proposition \ref{BkSmoothV1}.

\begin{proposition}Suppose $h^0 ( \cV_s \otimes \cE_t) = k$. The Zariski tangent space $T_{(s, t)} B^k ( \cV, \cE )$ is the annihilator of the image of
\[ {^t\kappa} \circ ( \tr_V, \tr_E ) \circ \mu \colon H^0 ( \cV_s \otimes \cE_t ) \otimes H^0 ( \Kc \otimes \cE_t^* \otimes \cV_s^* ) \to T^*_s S \oplus T^*_t T . \]
In particular, if $S \times T$ is smooth at $(s, t)$, then $B^k ( \cV, \cE )$ is smooth and of the expected dimension at $(s, t)$ if and only if ${^t\kappa} \circ ( \tr_V, \tr_E ) \circ \mu$ is injective. \label{BkSmoothV2} \end{proposition}

\noindent The most important corollary of this proposition is:

\begin{corollary} Let $V$ be a bundle of rank $r$ and degree $d$. Suppose $E \in U(n,e)$ satisfies $h^0 ( V \otimes E ) = k$. The twisted Brill--Noether locus $\bv$ is smooth and of the expected dimension at $E$ if and only if
\begin{equation} \tr_E \circ \mu \colon H^0 ( V \otimes E ) \otimes H^0 ( \Kc \otimes E^* \otimes V^*) \to H^0 ( \Kc \otimes \End (V \otimes E)) \to H^0 ( \Kc \otimes \End E) \label{PetriEtrace} \end{equation}
is injective. \label{bvSmooth} \end{corollary}

\begin{proof} This follows from Proposition \ref{BkSmoothV2} applied to the family $\cV$ consisting of the single bundle $V$ and a local universal family $\cE$ for $E$, together with the fact that the Kodaira-Spencer map for the family $\cE$ at $E$ is an isomorphism.
\end{proof}

It will be convenient to make the following definition.

\begin{definition} For fixed $V$, write $\mu_E$ for the composed map $\tr_E \circ \mu$ in (\ref{PetriEtrace}). We say that $V$ is \textsl{Petri $E$-trace injective} if $\mu_E$ is injective. 
 If the trace map
\[ \mu_{\Oc} \colon H^0 ( V ) \otimes H^0 ( \Kc \otimes V^* ) \ \to \ H^0 ( \Kc ) \]
is injective, we say that $V$ is \textsl{Petri trace injective}. \end{definition}

Next, as it will be central to several proofs, let us state \cite[Theorem 1.1]{TiB14} precisely.

\begin{theorem} \label{TiB} Let $C$ be a general curve and $V$ a general vector bundle over $C$. Then for any degree $e$ and $L \in \Pic^e (C)$, the Petri trace map
\[ \mu_L \colon H^0 ( V \otimes L ) \otimes H^0 ( \Kc\otimes L^{-1} \otimes V^*) \to H^0 ( \Kc ) \]
is injective. \end{theorem}

\noindent This motivates another definition.

\begin{definition} A vector bundle $V$ is \textsl{Petri general} if $V$ is Petri $L$-trace injective for all $L \in \Pic(C)$. \end{definition}

\begin{remark} A curve $C$ is \textsl{Petri} in the usual sense if and only if $\Oc$ is a Petri general vector bundle. It is well known that the general curve $C$ is a Petri curve.
\end{remark} 


\subsection{A partial desingularisation of $B^k ( \cV, \cE )$} \label{PartialDesing}

Here we generalise the construction $G^r_d(C)$ of \cite[IV.4]{ACGH} to twisted Brill--Noether loci.

For families $\cV$ and $\cE$, let us fix the effective divisor $D$ in \S \ref{BNfamily} and recall the complex $\gamma \colon K^0 \to K^1$. As $K^0$ is locally free, we have a Grassmannian bundle $\pi \colon \Gr(k, K^0) \to S \times T$. We define
\[ G^k ( \cV, \cE ) \ := \ \{ \Lambda \in \Gr ( k, K^0 ) : \gamma|_\Lambda = 0 \} . \]
This is a parameter space for triples $(V, E, \Lambda)$ where $\Lambda$ is a $k$-dimensional subspace of $H^0 ( V \otimes E)$. It seems natural to call such a triple a ``twisted coherent system'', but we do not pursue questions of moduli or stability here. When the family $\cV$ consists of a single vector bundle $V$, we write also $G^k(V,\cE)$.

Clearly, $\pi ( G^k ( \cV, \cE ) ) = B^k ( \cV, \cE)$ and $\pi^{-1} ( s, t) = \Gr (k, H^0 (\cV_s \otimes \cE_t))$. Let us describe the Zariski tangent spaces of $G^k ( \cV, \cE)$ at $(s,t)$.

\begin{proposition} \label{GkVE} \begin{enumerate} \item[(1)] Suppose $h^0 ( \cV_s \otimes \cE_t ) \ge k$. There is an exact sequence
\begin{multline} 0 \to \Hom \left( \Lambda, H^0 ( \cV_s \otimes \cE_t ) / \Lambda \right) \to T_{(\Lambda, s, t)} G^k ( \cV, \cE) \xrightarrow{d \pi} T_s S \oplus T_t T \to \\ \Hom \left( \Lambda, H^1 ( \cV_s \otimes \cE_t) \right) \label{GkTangSp} \end{multline}
where the last map is defined by $\cup \circ c \circ \kappa$ as in (\ref{TangSpSeq}), followed by restriction to $\Lambda$. Moreover, $\Image ( d\pi ) \ = \ \Image \left( {^t\kappa} \circ ( \tr_V , \tr_E ) \circ \mu \right)^\perp$.
\item[(2)] The locus $G^k ( \cV, \cE)$ is smooth and of dimension (\ref{expdimBk}) at $(s,t,\Lambda)$ if and only if the restricted map
\[ {^t\kappa} \circ (\tr_V , \tr_E ) \circ \mu \colon \Lambda \otimes H^0 ( \Kc \otimes E^* \otimes V^*) \to H^0(K_C\otimes \End V) \oplus H^0 ( \Kc \otimes \End E) \]
is injective. 
\item[(3)] In particular, if ${^t\kappa} \circ (\tr_V , \tr_E ) \circ \mu$ is injective and $h^0 ( \cV_s \otimes \cE_t ) > k$, then $G^k ( \cV, \cE )$ is a desingularisation of $B^k ( \cV, \cE)$ in a neighbourhood of $(s, t)$. \end{enumerate} \label{Gk} \end{proposition}

\begin{proof} Statements (1) and (2) are proven in the same way as \cite[Proposition IV.4.1 (ii)-(iii), p. 187]{ACGH}, and clearly (3) follows from (2). \end{proof}

\subsection{A sufficient condition for the existence of good components} \label{stratification}

Note that 
\[S\times T\supset B^1(\cV,\cE)\supset B^2(\cV,\cE)\supset\cdots\supset B^k(\cV,\cE)\] is a stratification of $B^k(\cV,\cE)$ by closed subsets. The following proposition makes use of this stratification.

\begin{proposition} \label{OneImpliesAll} Suppose $S\times T$ is smooth at $(s,t)$ and for some $k' \ge \chi$ there exists $(s, t) \in B^{k'}(\cV, \cE)$ such that $h^0 (\cV_s \otimes \cE_t) = k'$ and ${^t\kappa} \circ (\tr_{\cV_s} , \tr_{\cE_t}) \circ \mu$ is injective. Then, for $\chi \le k \le k'$, the locus $B^k (\cV, \cE)$ contains a component which is generically smooth and of the expected dimension. \end{proposition}

\begin{proof} We prove this by descending induction on $k$. For $k=k'$, the result follows immediately from Proposition \ref{BkSmoothV1}. Now suppose $\chi\le k<k'$ and that the proposition holds for $B^{k+1}(\cV,\cE)$. Then, there exists $(V,E)\in B^{k+1}(\cV,\cE)$ with $h^0(V\otimes E)=k+1$ and ${^t\kappa} \circ (\tr_{\cV_s} , \tr_{\cE_t}) \circ \mu$ injective. Now let $\Lambda$ be any $k$-dimensional linear subspace of $H^0(V\otimes E)$. Then, by Proposition \ref{Gk}, $G^k(\cV,\cE)$ is smooth of the expected dimension at $(V,E,\Lambda)$. Since $k\ge\chi$, it follows from \eqref{expdimBk} that every component of $G^k(\cV,\cE)$ has dimension greater than the dimension of $\pi^{-1}(B^{k+1}(\cV,\cE))$ at $(V,E)$. Hence, there exists a point $(V_1,E_1,\Lambda_1)$ of $G^k(\cV,\cE)$ in the neighbourhood of $(V,E,\Lambda)$ with $h^0(V_1\otimes E_1)=k$ and  ${^t\kappa} \circ (\tr_{V_1} , \tr_{E_1}) \circ \mu$ injective. Thus $B^k(\cV,\cE)$ is smooth of the expected dimension at $(V_1,E_1)$.
\end{proof}

This proposition illustrates a general principle that, from the existence of just one pair of bundles with good properties, one can obtain a detailed picture of the geometry of several of the strata. This will be used on a number of occasions later.

\subsection{The loci $\tB^k_{n,e}(V)$} \label{semistable} We end this section with a brief discussion of twisted Brill--Noether loci where strictly semistable bundles are admitted.

The space $U(r,d)$ is an open subset of the moduli space $\tU(r,d)$ of S-equivalence classes of semistable bundles of rank $r$ and degree $d$. We write $[E]$ for the S-equivalence class of a semistable $E$ and $\operatorname{gr}E$ for the graded bundle associated to $E$; $\operatorname{gr}E$ depends only on $[E]$. The definition of $\bv$ is extended to include semistable bundles by setting
\[ \tB^k_{n, e} (V) \ := \ \{ [E] \in \tU(n, e) :  h^0 ( C, V \otimes \operatorname{gr}E ) \ge k \} . \]

Furthermore; it was noted in \cite[\S 4]{GTiB} that, for $r \ge 2$, the locus $B^k_{r,0}$ is empty, but $\tB^k_{r,0}\ne\emptyset$ for $k\le r$.
 We can generalise this example to twisted Brill--Noether loci. Suppose $V \in U(r,d)$ is a stable bundle, $n>r$ and $re+nd=0$. Then, according to Proposition \ref{prop23} (2), $B^k_{n,e}(V)=\emptyset$ for all $k$. However, taking $E=V^*\oplus F$, where $F$ is semistable and $\mu(F) = \mu(E)$, we see that $[E]\in\tB^1_{n,e}(V)$.

\section{Two irreducibility results} \label{irred}

Here we will give some applications of the machinery assembled in the previous section.

\subsection{Rank one twisted Brill--Noether loci} If $C$ is a Petri curve, $B^k_{1,e} = W_e^{k-1}(C)$ is irreducible whenever $\rho^k_{1,e} = g - k(k - e + g - 1) \ge 1$. 

Let $\cP \to \Pic^e (C) \times C$ be a Poincar\'e bundle.

\begin{theorem} \label{IrrPGfamilies} Let $\cV \to S \times C$ be a family of Petri general vector bundles of rank $r$ and degree $d$ parametrised by a smooth irreducible base $S$. Assume that
\[ \rho^k_{1, e, r, d} \ = \ g - k(k - (d+re) + r(g-1)) \ \ge \ 1 . \]
Then the locus $B^k (\cV , \cP ) \subseteq S \times \Pic^e (C)$ is an irreducible variety of dimension $\dim S + \rho^k_{1, e, r, d}$ which is singular precisely along $B^{k+1} (\cV, \cP)$. \end{theorem}

\begin{proof} The fibre of $B^k ( \cV, \cP )$ over each $s \in S$ is exactly $B^k_{1, e} (\cV_s)$. Since $\rho^k_{1,e,r,d} \ge 1$, by Theorem \ref{LTiB} (1) and (2) this fibre is nonempty and connected. As $S$ is irreducible, it follows that $B^k ( \cV, \cP )$ is connected. As the fibres of $G^k ( \cV, \cP ) \to B^k ( \cV, \cP)$ are Grassmannians, $G^k ( \cV , \cP)$ is also connected.

Furthermore, as $\cV_s$ is Petri general for all $s$, by Proposition \ref{Gk}(2) in fact $G^k ( \cV, \cP )$ is smooth. Therefore $G^k (\cV, \cP)$ is irreducible. As $B^k (\cV, \cP )$ is the image of $G^k ( \cV, \cP)$ by a morphism, $B^k ( \cV, \cP )$ is also irreducible.

The last statement follows from Petri generality and Proposition \ref{BkSmoothV2}. \end{proof}

\begin{remark} Suppose that $C$ is a general curve. Since, by Theorem \ref{TiB}, a general bundle $V$ in $U(r, d)$ is Petri general, in particular $B^k_{1, e}(V)$ is irreducible for general $V$. Thus we recover \cite[Th\'eor\`eme 1.2]{Hir2}. \end{remark}

\subsection{Irreducibility of Petri trace injective loci} Here we give an application to ``nontwisted'' Brill--Noether loci.

\begin{theorem} \label{PTIIrrComp} Suppose $C$ is general and $g \ge \rho^k_{1, 0, r, d} = g - k(k-d+r(g-1)) \ge 1$. Then there is a unique irreducible component $\Bkg$ of the Brill--Noether locus $B^k_{r, d}$ containing the locus
\begin{equation} \{ V \in B^k_{r, d} : \bar{\mu} \colon \Lambda \otimes H^0 ( \Kc \otimes V^* ) \to H^0 ( \Kc ) \hbox{ is injective for some } \Lambda \in \Gr (k, H^0 ( V ) ) \} \label{BkrdPTI} \end{equation}
In particular, the locus of Petri trace injective bundles in $B^k_{r,d}$ is irreducible. \end{theorem}

\begin{proof}
By \cite[Proposition 2.6]{NR75}, there exists a smooth irreducible variety $\cM$ admitting a Poincar\'e family $\cV \to \cM \times C$ such that every stable bundle of rank $r$ and degree $d$ over $C$ is represented in $\cM$. Throughout, we will write $\tV$ for a point of $\cM$ lying over $V \in U (r, d)$.

Set $e = 0$ and let $\cP \to \Pic^0 (C) \times C$ be a Poincar\'e bundle. We consider the locus $G^k ( \cV, \cP ) \to \cM \times \Pic^0 (C)$ defined in \S \ref{PartialDesing}. Define
\[ \cG^k \ := \ \{ (\tV, L, \Lambda) \in G^k ( \cV, \cP ) : \mu_L \colon \Lambda \otimes H^0 ( \Kc \otimes L^{-1} \otimes V^* ) \to H^0 ( \Kc ) \hbox{ is injective} \} . \]
By Theorem \ref{LTiB} (3), (4) and by hypothesis, this is non-empty. By Proposition \ref{GkVE} (2), it is precisely the smooth locus of $G^k ( \cV, \cP )$.

Let $t \colon G^k ( \cV, \cP ) \to U (r, d)$ be the morphism given by $t(\tV, L, \Lambda) = V \otimes L$. By definition, $t ( \cG^k )$ is the locus (\ref{BkrdPTI}). By Brill--Noether theory, $t(\cG^k)$ has the expected dimension, so is a union of components of $B^k_{r,d}$. Thus it will suffice to show that $\cG^k$ is irreducible.

Let $p$ be the projection map $G^k ( \cV, \cP ) \to \cM \times \Pic^0 (C) \to \cM$. By Theorem \ref{LTiB} (1), this is surjective. We claim that there is a unique irreducible component of $\cG^k$ which dominates $\cM$. Suppose that $X_1$ and $X_2$ are components of $\cG^k$ such that $p(X_1)$ and $p (X_2)$ are both dense in $\cM$. 
 Let $\tV$ be a general point of $p(X_1) \cap p(X_2)$. The fibre $p^{-1} (\tV)$ is identified with
\begin{equation} G^k ( \tV, \cP ) \ = \ \{ ( L, \Lambda ) : \Lambda \in \Gr ( k , H^0 ( \tV \otimes L )) \} , \label{GkV} \end{equation}
here viewing $\tV$ as a singleton family. By hypothesis and by Theorem \ref{TiB}, we may assume $V$ is Petri general. Thus, by the proof of Theorem \ref{IrrPGfamilies}, the locus $G^k ( \tV, \cP )$ is irreducible. Hence by semicontinuity of fibre dimension, $p^{-1} (\tV)$ is generically contained in both $X_1$ and $X_2$. In particular, $X_1 \cap X_2$ is nonempty. Since $\cG^k$ is smooth, the only possibility is that $X_1 = X_2$.

Therefore, to conclude, it will suffice to show that the restriction of $p$ to \emph{any} component $X$ of $\cG^k$ is dominant. To see this: Let $(\tV, L, \Lambda)$ be a point of $X$. 
 By Proposition \ref{GkVE} (2) applied to the locus (\ref{GkV}), we have
\[ \dim_{(\tV, L, \Lambda)} (p|_X)^{-1} \left( p ( \tV, L, \Lambda ) \right) \ \le \ \rho^k_{1, 0, r, d} \ = \ g - k ( k - d + r(g-1) ) . \]
On the other hand, again by Proposition \ref{GkVE} (2), since $X$ is smooth, we have
\[ \dim_{(\tV, L, \Lambda)} X \ = \ \dim \cM + g - k ( k - d + r(g-1) ) \]
Thus $\dim ( p(X) ) \ge \dim ( \cM )$. As $\cM$ is irreducible, $p(X)$ is dense in $\cM$. This completes the proof. \end{proof}

\section{Nonemptiness of $\bv$ and $\tB^k_{n,e}(V)$} \label{chapternonemptiness}

We now prove Theorem \ref{nonemptiness} using the method of \cite{Mer}. This is very straightforward; the necessary ingredients already exist by Theorem \ref{LTiB} and the results in \cite{Mer} on stability of elementary transformations. We note that Mercat's construction was used in a similar way in \cite[\S 6]{BBPN} to show the nonemptiness of certain moduli spaces of coherent systems.


\begin{proof}[Proof of Theorem \ref{nonemptiness}] By hypothesis, $B^{k_0}_{1, e_0}(V)$ is of positive dimension. Thus we can find mutually nonisomorphic line bundles $L_1 , \ldots , L_n$ of degree $e_0$ such that $h^0 (V \otimes L_i ) \ge k_0$ for each $i$. Let
\[ 0 \to \bigoplus_{i=1}^n L_i \to E \to \tau \to 0 \]
be a general elementary transformation. We have the cohomology sequence
\[ 0 \to \bigoplus_{i=1}^n H^0 ( V \otimes L_i ) \to H^0 (V \otimes E) \to \cdots \]
so $h^0 (V \otimes E) \ge nk_0$. If $\deg \tau \ge 1$, then it follows easily from \cite[Th\'eor\`eme A.5]{Mer} that $E$ is a stable vector bundle, and the result follows for $e > ne_0$. If $\tau = 0$ then $E = \oplus_{i=1}^n L_i$ gives an element of $\tB^{nk_0}_{n, ne_0}$, and the result follows also in the case $e = ne_0$. \end{proof}

\noindent In the next sections, we will refine this statement in some cases.

\section{Generatedness of Petri general bundles} \label{genPTI}

To prove the existence of components of $B^k_{n,e} (V)$ which are generically smooth and of the expected dimension, the need will emerge to show the existence of bundles $W$ of rank $r \ge 2$ and degree $d$ satisfying the following conditions:
\begin{enumerate}
\item[(1)] $W$ is Petri trace injective. Equivalently, $\Kc \otimes W^*$ is Petri trace injective.
\item[(2)] $h^0 (W) = k \ge 1$.
\item[(3)] $\Kc \otimes W^*$ is generically generated.
\end{enumerate}

\begin{remark} \label{PTIbounds} Note that the above conditions give strong bounds on $d$ and $k$. By (1) and (2), we have $d \ge r(g-1) + 1 - g  = (r-1)(g-1)$. 
 Moreover, (3) implies that $h^0 (\Kc \otimes W^*) = h^1 (W) = k - \chi(W) \ge r$, whence by (1) we see that $\deg(K_C\otimes W^*) \ge r(g-1) + r - \frac{g}{r}$, so $d \le r(g-2) + \frac{g}{r}$. In summary,
\begin{equation} (r-1)(g-1)\ \le \ d \ \le \ r(g-2) + \frac{g}{r} . \label{NecessaryForPTIgengen} \end{equation}
Values of $d$ satisfying \eqref{NecessaryForPTIgengen} exist if and only if $r\le g$.
\end{remark}

\subsection{The construction}
Write $g = rl+r_0$ where $l, r_0$ are integers with $0 \le r_0 < r$. Let $D_0$ be an effective divisor of degree $r_0$ such that $h^0(K_C(-D_0))=h^0(K_C)-r_0=rl$.  (If $r_0 = 0$, take $D_0 = 0$.) Set $N := \Kc(-D_0)$ and let $D_1, \ldots , D_r$ be distinct effective divisors of degree $l$ such that
\begin{equation} h^0\left(N\left(-\sum D_j\right)\right)=h^0(N)-rl=0. \label{DiSpan} \end{equation}
For $1 \le i \le r$, set $M_i \ := \ N \left( - \sum_{j \ne i} D_j \right)$. By (\ref{DiSpan}), we have $h^0 (M_i ) = h^0 (N) - (r-1)l = l$ for each $i$.

\begin{lemma} \ 
\begin{enumerate} 
\renewcommand{\labelenumi}{(\arabic{enumi})}
\item $H^0 (N) \cong \bigoplus_{i=1}^r H^0 (M_i)$.
\item $H^1 ( M_i ) \cong H^1 (N) = \K$.
\item The bundles $M_1 , \ldots , M_r$ are mutually nonisomorphic.
\end{enumerate} \label{SectionsOfMi} \end{lemma}
\begin{proof} (1) 
We have inclusions $H^0 ( M_i ) \hookrightarrow H^0 ( N )$ for all $i$. It is easy to see that
\[H^0(M_i)\cap \left( H^0(M_1)+\cdots+H^0(M_{i-1})+H^0(M_{i+1})+\cdots + H^0(M_r) \right)=0.\]
So $\oplus_{i=1}^rH^0(M_i)\subseteq H^0(N)$. Since the dimensions agree, we obtain (1).

(2) Calculating values for degrees and $h^0$, this follows from Riemann-Roch.

(3) Suppose $M_{i_1} \cong M_{i_2}$; that is,
\[ N \left( - \sum_{j \ne i_1} D_j \right) \ \cong \ N \left( - \sum_{j \ne i_2} D_j \right) . \]
Tensoring both sides with $N^{-1} \left( \sum_{j=1}^r D_j \right)$, we obtain $\Oc ( D_{i_1} ) \cong \Oc (D_{i_2})$. As $D_{i_1} \ne D_{i_2}$ as divisors, in particular $h^0 ( \Oc (D_{i_1})) \ge 2$. But then $h^0(K_C(-D_{i_1}))\ge g-l+1$ and $h^0(M_{i_2})\ge g-l+1-r_0-(r-2)l=l+1$, a contradiction. This proves (3).
\end{proof}

Now write $G := \bigoplus_{i=1}^r M_i$. Consider elementary transformations
\begin{equation} 0 \to W \to G \to T \to 0 \label{WT} \end{equation}
where $T$ is a torsion sheaf of degree $(r-1)l + m$ with $0 \le m \le l-1$. To ease notation, we write $t := (r-1)l + m$. The set of such $W$ is parametrised by the Quot scheme $\Quot^{0, t}(G)$, an irreducible variety of dimension $rt$. 

\begin{lemma} For general $W \in \Quot^{0, t}(G)$, the map $H^0 (G) \to H^0 (T)$ is surjective. In particular, $h^0 ( W ) = l-m$ and $H^0 (\Kc \otimes W^*) \cong H^0 (\Kc \otimes G^*)$. \label{CoboundarySurjective} \end{lemma}
\begin{proof} Since the surjectivity condition is open, it is sufficient to prove the existence of one bundle $W$ with the required property. Suppose first that $t=1$ and let $T=\mathbb{K}_p$, where $p$ is a point at which some section of $G$ is nonzero. Then we can find a surjection $G\to T$ such that $H^0(G)\to H^0(T)$ is surjective. Repeating this argument, we obtain the result  by induction on $t$.
\end{proof}

We want one more generality condition on $W$. For $1 \le i \le r$, write $\hG_i := \bigoplus_{j \ne i} M_j$, and consider the sheaf
\[ W \cap \hG_i \ = \ \Ker \left( W \to G \to M_i \right) . \]

\begin{lemma} If $W$ is sufficiently general in $\Quot^{0,t}(G)$, then $h^0 (W \cap \hG_i ) = 0$. \label{WcaphGi} \end{lemma}
\begin{proof} Since the condition $h^0(W\cap\hG_i) = 0$ is open, it is sufficient to find one example of an elementary transformation \eqref{WT} for which this property holds. For this, consider elementary transformations
\[0\to W_1\to\hG_i\to T\to0,\]
where $T$ is as in \eqref{WT}. The same argument as for Lemma \ref{CoboundarySurjective} shows that, for general $W_1$, we have $h^0(W_1)=\max\{0,-m\}=0$. Now take $W=W_1\oplus M_i$.
\end{proof}

\begin{proposition}  A general elementary transformation $W \in \Quot^{0, t}(G)$ is stable and Petri trace injective, and has $\Kc \otimes W^*$ generically generated. \end{proposition}

\begin{proof} As the $M_i$ are mutually non-isomorphic by Lemma \ref{SectionsOfMi} (3), by \cite[Th\'eor\`eme A.5]{Mer} the bundle $W$ is stable for general $T$. By Lemma \ref{CoboundarySurjective}, we have $H^0 ( \Kc \otimes W^*) \cong H^0 ( \Kc \otimes G^*)$. From Lemma \ref{SectionsOfMi} (2), it then follows that $\Kc \otimes W^*$ is generically generated.

We describe $H^0 ( \Kc \otimes W^*)$ more explicitly. For $1 \le i \le r$, let $t_i'$ be a generator of $H^0 ( \Kc\otimes M_i^{-1})$, and write $t_i$ for the image of $t_i'$ in $H^0 ( \Kc \otimes W^* )$. Thus we obtain a splitting $H^0 ( \Kc \otimes W^*) = \bigoplus_{i=1}^r \K \cdot t_i$.
Therefore, we can write any element of $H^0 (W) \otimes H^0 ( \Kc \otimes W^*)$ in the form
\begin{equation} \sum_{i=1}^r \sigma_i \otimes t_i \label{WKWDtensor} \end{equation}
where $\sigma_i = ( s_{i,1} , s_{i, 2} , \ldots , s_{i, r})$ is a section of $G$ belonging to $W$, that is, lying in the kernel of $H^0 (G) \to H^0 (T)$. The Petri trace is then given by
\[ \bar{\mu_W} \left( \sum_{i=1}^r \sigma_i \otimes t_i \right) \ = \ \sum_{i=1}^r \mu_{M_i} ( s_{i,i} \otimes t_i' ) \ \in \ H^0 ( \Kc ) . \]

To analyse this, note that, by Lemma \ref{SectionsOfMi} (2), the homomorphism $M_i\to N$ induces an isomorphism $H^0(K_C\otimes N^{-1})\to H^0(K_C\otimes M^{-1})$. It follows that there is a commutative diagram
\[ 
\xymatrix{ 
\bigoplus_{i=1}^r H^0 (M_i) \otimes H^0 ( \Kc\otimes M_i^{-1}) \ar[r]^-{\oplus \mu_{M_i}} \ar[d] & \bigoplus_{i=1}^r H^0 ( \Kc ) \ar[d]^{\mathrm{sum}} \\
 H^0 ( N ) \otimes H^0 (\Kc\otimes N^{-1}) \ar[r]^{\quad\quad\quad \mu_N} & H^0 ( \Kc ), } 
 \]
where the lefthand vertical map is an isomorphism by Lemma \ref{SectionsOfMi} (1) and (2). Since $h^0(K_C\otimes N^{-1})=1$ by Lemma \ref{SectionsOfMi} (2), $\mu_N$ is injective. By commutativity, the composed map $\bigoplus_{i=1}^r H^0 (M_i) \otimes H^0 ( \Kc\otimes M_i^{-1} ) \to H^0 ( \Kc )$ is injective.

This means that a tensor of the form (\ref{WKWDtensor}) has trace zero only if $s_{i,i} = 0$ for all $i$. Thus $\sigma_i$ belongs to the subsheaf $W \cap \hG_i$. But by Lemma \ref{WcaphGi}, for general $W\in\Quot^{0, t}(G)$, $h^0(W \cap \hG_i)=0$ and $\sigma_i=0$ for all $i$. This completes the proof. \end{proof}

\begin{corollary}\label{corgen} Let $g$, $r$, $l = \left\lfloor \frac{g}{r} \right\rfloor$ and $r_0$ be as above. For $0 \le m \le l-1$, there exists a Petri trace injective bundle $W$ of rank $r$ and degree $r(g-2)+l-m$ with $h^0 (W) = l-m$ and such that $\Kc \otimes W^*$ is generically generated. \end{corollary}


\begin{corollary} Let $C$ be a general curve of genus $g$. Suppose $0 \le m \le l-1$ and $d = r(g-2) + l - m$. If $m=0$, suppose further that $g\not\equiv 0\bmod r$. Then $\left( B^{l-m}_{r, d} \right)_{\PTI}$ is irreducible and, if $W$ is a general element of $\left( B^{l-m}_{r, d} \right)_{\PTI}$ and $p$ is a general point of $C$, 
\begin{equation}\label{eqgen}h^0(W)=h^0(W(p))=l-m.
 \end{equation}\label{GenGen}\end{corollary}
\begin{proof}
The irreducibility of $\left( B^{l-m}_{r, d} \right)_{\PTI}$ follows from Theorem \ref{PTIIrrComp}. (The hypothesis $g \not\equiv 0 \mod r$ is required to ensure that the numerical hypothesis of Theorem \ref{PTIIrrComp} holds.) The property \eqref{eqgen} is open. It is therefore sufficient to find one example of a bundle $W$ with the stated property. For this, take $W$ as in Corollary \ref{corgen}. The statement \eqref{eqgen} is then equivalent to the injectivity of the coboundary map in the sequence
\[ 0 \to H^0 ( W ) \to H^0 ( W(p) ) \to W(p)|_p \to H^1 ( W) \to H^1 (W(p)) \to 0 . \]
By Serre duality, the coboundary map is injective if and only if the evaluation map
\begin{equation} H^0(K_C \otimes W^*) \ \to \ \Kc\otimes W^*|_p \label{EvalMap} \end{equation}
is surjective. By Corollary \ref{corgen}, this is true for general $p$.
\end{proof}

\section{Smoothness of twisted Brill--Noether loci} \label{smoothness}

In this section, we prove our main result Theorem \ref{main}. The major part of the section is concerned with proving the following technical proposition.

\begin{proposition} \label{HigherRank} Let $C$ be a general curve of genus $g$. Suppose $e_0$ and $k_0$ are integers satisfying
\begin{equation} g - k_0 \left( k_0 - (d + r e_0) + r(g-1) \right) \ \ge \ 1 \label{BNnumber} \end{equation}
and furthermore that there exists a bundle $W \in \left( B^{k_0}_{r, d+re_0} \right)_{\PTI}$ such that for general $p$ in $C$ we have $h^0(W)=h^0(W(p))=k_0$. Write $e = ne_0 + e_1$ where $1 \le e_1 \le n$. Then for general $V \in U(r, d)$ and for $re + nd - rn(g-1) \le k \le nk_0$, the twisted Brill--Noether locus $\bv$ has a component which is nonempty, generically smooth and of the expected dimension. \end{proposition}

We begin with a lemma which has applications to coherent systems (twisted or untwisted) as well as to twisted Brill-Noether loci. For this, let $V$ be any bundle of rank $r$ and degree $d$ and consider the Grassmannian bundle $G^{k_0}(V, \cP^{e_0})$, where $\cP$ is a Poincar\'e family on $\Pic^{e_0} (C) \times C$, which parametrises pairs $(L,\Lambda_0)$ with $L$ a line bundle of degree $e_0$ and $\Lambda_0\subset H^0(V\otimes L)$ a linear subspace of dimension $k_0$. Suppose further that $X$ is an irreducible component of $G^{k_0}(V, \cP^{e_0})$ which is generically smooth of dimension
\[\rho^{k_0}_{1,e_0}(V)=g - k_0 \left( k_0 - (d + r e_0) + r(g-1) \right) \ \ge \ 1.\]
Let $(L_1,\Lambda_1) , \ldots , (L_n,\Lambda_n)$ be points of $X$, and write $F := \bigoplus_{i=1}^n L_i$ and $\Lambda:=\bigoplus_{i=1}^n \Lambda_i$. We consider elementary transformations
\begin{equation}\label{elem} 0 \to F \to E \to \tau \to 0 \end{equation}
with $\tau$ a torsion sheaf of length $e_1$. 

\begin{lemma}\label{lemcs} Under the above conditions, let $(L_1,\Lambda_1) , \ldots , (L_n,\Lambda_n)$ be general points of $X$. Then the restricted Petri $E$-trace map
\begin{equation}\label{PetriEtracemap}
\mu_E:\Lambda\otimes H^0(K_C\otimes E^*\otimes V^*)\to H^0(K_C\otimes \End E)
\end{equation}
is injective.
\end{lemma}
\begin{proof} Consider first the restricted Petri $F$-trace map of $V$, given by
\begin{equation} \mu_F \colon \Lambda \otimes H^0 (\Kc \otimes F^* \otimes V^*) 
\to H^0 ( \Kc \otimes \End F) . \label{PetriFTraceMap} \end{equation}
Noting that $H^0 ( \Kc \otimes \End F) = \bigoplus_{i,j} H^0 ( \Kc\otimes L_j^{-1}\otimes L_i )$, we see that (\ref{PetriFTraceMap}) is the direct sum of the trace maps
\begin{equation} \mu_{i,j} \colon \Lambda_i \otimes H^0 ( \Kc \otimes L_j^{-1} \otimes V^*) \ \to \ H^0 ( \Kc\otimes L_j^{-1}\otimes  L_i) \label{ijTrace} \end{equation}
for $1 \le i \le n$ and $1 \le j \le n$. Thus $\mu_F$ is injective if and only if $\mu_{i,j}$ is injective for all $i, j$.

Write
\[ U \ := \ \{ (L,\Lambda_0) \in X: X \text{ \rm smooth at } (L,\Lambda_0), h^0 (V\otimes L) \text{ \rm takes its minimum value}\} , \]
which is a non-empty open subset of  $X$ by semicontinuity. We can assume that $(L_i,\Lambda_i)\in U$ for all $i$.

Now let $p$ be a point of $C$. For each $(L,\Lambda_0) \in U$, we have a commutative diagram
\[ \xymatrix{ \Lambda_0 \otimes H^0 ( \Kc\otimes L^{-1} \otimes V^*) \ar[r]^-{\mu_0} \ar@{=}[d] & H^0 ( \Kc ) \ar[d]^\wr \\ \Lambda_0 \otimes H^0 ( \Kc\otimes L^{-1} \otimes V^*) \ar[r]^-{\mu'_0} & H^0 ( \Kc (p) ) } \]
where, in the second line, $\Lambda_0$ is regarded as a subspace of $H^0(V\otimes L(p))$ and the horizontal arrows are trace maps. Since $X$ is smooth at $(L,\Lambda_0)$, $\mu_0$ is injective. Hence so is $\mu'_0$.

Next, let $A$ be the direct image sheaf over $U \times U$ whose fibre at $((L,\Lambda_0), (N,\Lambda_0'))$ is
\[ \Lambda_0 \otimes H^0 ( \Kc \otimes N^{-1} \otimes V^*) . \]
Since $H^0(K_C\otimes N^{-1}\otimes V^*)$ is constant on $U$, this is locally free.
 Furthermore, let $B$ be the direct image sheaf over $U$ whose fibre at $((L,\Lambda_0), (N,\Lambda_0'))$ is $H^0 ( \Kc\otimes N^{-1}\otimes L(p))$. This is locally free of rank $g$.

Write $\tilde{\mu} \colon A \to B$ for the globalised Petri trace map whose restriction to $((L,\Lambda_0), (N,\Lambda_0'))$ is the trace map
\[\Lambda_0\otimes H^0(K_C\otimes N^{-1}\otimes V^*)\to H^0(K_C\otimes N^{-1}\otimes L(p)).\]
As $\tilde{\mu}|_{((L,\Lambda_0), (L,\Lambda_0))}$ coincides with $\mu'_0$ above, $\tilde{\mu}$ is injective on a non-empty open subset $U'$ of $U \times U$. We can suppose that $((L_i,\Lambda_i), (L_j,\Lambda_j))\in U'$ for all $i,j$, so that $\tilde{\mu}|_{((L_i,\Lambda_i), (L_j,\Lambda_j))}$ is injective for all $i, j$. 
As $\tilde{\mu}|_{((L_i,\Lambda_i), (L_j,\Lambda_j))}$ factors through the trace map (\ref{ijTrace}), the latter is also injective. This completes the proof that $\mu_F$ is injective.

To see that $\mu_E$ is injective, we note that $K_C\otimes E^*\otimes V^*\subset K_C\otimes F^*\otimes V^*$ and consider the diagram of cohomology spaces
\[ \xymatrix{\Lambda\otimes H^0 ( \Kc \otimes E^* \otimes V^* ) \ar@{=}[r]\ar@{^{(}->}[d]_a \ar[dr]_c & \Lambda\otimes H^0 ( \Kc \otimes E^* \otimes V^* )\ar[dr]^{\mu_E} & \\
\Lambda\otimes H^0( \Kc \otimes F^* \otimes V^* ) \ar@{^{(}->}[d]_{\mu_F}  \ar[dr]^b & H^0 ( \Kc \otimes E^* \otimes F) \ar@{^{(}->}[r]^d \ar@{^{(}->}[d]_f & H^0 (\Kc \otimes \End E) \\ \Lambda\otimes H^0(K_C\otimes \End F)\ar[r]^-{\sim}& \bigoplus_{i, j} H^0 (\Kc\otimes L_j^{-1}\otimes L_i ). & } \]
We have already seen that $\mu_F$ is injective. Hence first $b$, then $c$, then $\mu_E=d\circ c$ are all injective.
\end{proof}

\begin{proof}[Proof of Proposition \ref{HigherRank}]
Suppose that the hypotheses of Proposition \ref{HigherRank} hold and let $V \in U(r, d)$ be general. By Proposition \ref{OneImpliesAll}, it suffices to exhibit a stable bundle $E \in U(n, e)$ with $h^0 (V \otimes E) = nk_0$ and such that $V$ is Petri $E$-trace injective. For this, we use the construction of \eqref{elem}, where we now assume that this elementary transformation is general and the bundles $L_i$ are all distinct. It then follows from  \cite[Th\'eor\`eme A.5]{Mer} that $E$ is stable.

Note next that, since $h^0(W(p))=k_0$, we must have $k_0>d+re_0-r(g-1)$, so $\rho_{1,e_0}^{k_0}(V)<g$. By Theorem \ref{LTiB}, it follows that the locus $B^{k_0}_{1, e_0}(V)$ is non-empty and irreducible of dimension $\rho^{k_0}_{1, e_0}(V) \ge 1$ (by \eqref{BNnumber}), and
\begin{equation} h^0 (V \otimes L) = k_0 \hbox{ for general } L \in B^{k_0}_{1,e_0}(V) . \label{hzeroconstant} \end{equation}
By Theorem \ref{TiB}, we may assume also that $V$ is Petri general. By Proposition \ref{PTIIrrComp}, $V\otimes L$ is a general point of $\left( B^{k_0}_{r, d+re_0} \right)_{\PTI}$. It now follows from the hypotheses of Proposition \ref{HigherRank} that 
\begin{equation}\label{eqLp}h^0(V\otimes L(p))=k_0.
\end{equation}

Now let $L_1 , \ldots , L_n$ be general points of $B^{k_0}_{1, e_0}(V)$, and write $F := \bigoplus_{i=1}^n L_i$. We can assume that $h^0(V\otimes L_i)=k_0$ for all $i$, so that $h^0(V\otimes F)=nk_0$. Now note that the condition $h^0(V\otimes E)=nk_0$ is an open condition, so it is sufficient to exhibit a single elementary transformation \eqref{elem} satisfying this condition. In fact, by \eqref{eqLp}, for general $p_i\in C$, we can take
\[ E_0 \ := \ \left( \bigoplus_{i=1}^{e_1} L_i(p_i) \right) \oplus \left( \bigoplus_{i=e_1 + 1}^n L_i \right). \]

Finally, we note that $B^{k_0}_{1,e_0}(V)$ is irreducible of the expected dimension by Theorem \ref{LTiB}, and moreover $B^{k_0+1}_{1,e_0}(V)$ is of the expected dimension. It follows that $G^{k_0}(V,\mathcal{P}^{e_0})$ is also irreducible of the expected dimension. Now apply Lemma \ref{lemcs} with $X=G^{k_0}(V,\mathcal{P}^{e_0})$  and  $\Lambda_i=H^0(V\otimes L_i)$ for all $i$. It follows that $V$ is Petri $E$-trace injective.
\end{proof}

\noindent Now we can prove our main result on smoothness and dimension of twisted Brill--Noether loci.

\begin{proof}[Proof of Theorem \ref{main}] A straightforward computation shows that the numerical hypotheses and Corollary \ref{GenGen} imply that the hypotheses of Proposition \ref{HigherRank} are satisfied.  
 \end{proof}

\begin{remark} By \cite[Theorem 2.10]{CTiB}, if $V$ is general in $U(r, d)$ then for any $E \in \bv$ the bundle $V \otimes E$ is stable. \end{remark}

\begin{remark} \begin{enumerate}
\item[(1)] According to \cite[Theorem 0.3]{RTiB}, if $0\le\rho^1_{n,e}(V)\le n^2(g-1)+1$, then every component of $B^1_{n,e}$ has dimension $\rho^1_{n,e}(V)$. The authors do not require the more stringent numerical conditions of our Theorem \ref{main}. Our result can be seen as a partial generalisation of \cite[Theorem 0.3]{RTiB}, although we do not show that every component of $\bv$ has the expected dimension.
\item[(2)] It seems reasonable to conjecture that the hypotheses of Proposition \ref{HigherRank} are satisfied in more cases than those covered in Theorem \ref{main}. The main obstacle to generalising the theorem is to show the generic generatedness of a general bundle in $\left( B^{k_0}_{r, d+re_0} \right)_\PTI$ in more cases. (Theoretical bounds can be deduced from (\ref{NecessaryForPTIgengen}).) \end{enumerate} \end{remark}

\begin{remark} For general $C$, the bundle $\Oc$ is Petri general and Petri trace injective. The above proofs are therefore valid and we recover \cite[Theorem 1.1]{CMTiB} (see also \cite{TiB91}).
\end{remark}

\section{Tangent cones of twisted Brill--Noether loci} \label{TangentCones}

Suppose $Y \subset X$ are varieties, and let $x \in Y$ be a smooth point of $X$. Recall that the tangent cone $\bT_x Y$ to $Y$ at $x$, set-theoretically, is
\[ \{ v \ \in \ T_x X : v \hbox{ is tangent to a smooth arc in } Y \} .\]
Generalising theorems of Kempf \cite{Kem} and Laszlo \cite{Lasz}, in \cite{CTiB} the theory of determinantal varieties is used to describe the tangent cones to $B^k_{r, d}$ at points where the appropriate Petri maps are injective. In \cite[Remark 2.8]{CTiB}, it is noted that the same approach can be used to describe $\bT_E \bv$, which is a subvariety of $T_E U(n, e) = H^1 ( \End E )$. In the following proposition, we follow up this remark and use Theorem \ref{main} to give some situations in which it applies.

\begin{proposition} \label{TangentCone} Let $V\in U(r, d)$ and suppose that $E \in B^k_{n, e} (V)$ with $k\ge\chi:=\chi(V\otimes E)$ and $\mu_E$ injective.
\begin{enumerate}
\renewcommand{\labelenumi}{(\arabic{enumi})}
\item The tangent cone $\bT_E \bv$ is Cohen-Macaulay, reduced and normal.
\item If $s_1 , \ldots , s_{h^0 (V \otimes E)}$ and $t_1 , \ldots , t_{h^1 ( V \otimes E )}$ are bases for $H^0 ( V \otimes E )$ and $H^0 ( \Kc \otimes E^* \otimes \nolinebreak V^* )$ respectively, then the ideal of $\bT_E \bv$ is generated by the minors of size $( h^0 ( V \otimes E ) - k + 1 ) \times ( h^0 ( V \otimes E) - k + 1 )$ of the matrix whose $(i, j)$th entry is
\[ \begin{pmatrix} \mu_E ( s_i \otimes t_j ) \end{pmatrix} , \] 
an element of $H^0 ( \Kc \otimes \End E ) = H^1 ( \End E )^*$. 
\item The degree of $\bT_E \bv$ is
\[ \prod_{h = 0}^{k-1} \frac{(h^1 ( V \otimes E ) + h) ! \cdot h !}{(h^0 ( V \otimes E ) - k + h)! \cdot ( k + h - \chi ( V \otimes E ))!} .\]
\item As a set, $\bT_E \bv$ is the union
\[ \bigcup_{\Lambda \in \Gr ( k, H^0 (V \otimes E))} \mu_E ( \Lambda \otimes H^0 ( \Kc \otimes E^* \otimes V^*))^\perp \ \subseteq \ H^1 ( \End E ) . \]
\end{enumerate}
\end{proposition}


\begin{proof} If $k = \chi$, then $\bT_E \bv = H^1 ( \End E )$, and the various parts of the proposition follow easily. (In particular, the formula in (4) yields the required degree $1$.) Thus, in what follows, we will assume $k > \chi$.

 As before, let $\cE \to \tUne \times C$ be a Poincar\'e bundle, where $\tUne \to U(n, e)$ is a suitable \'etale cover. Fix a point in $\tUne$ lying over $E$, and, abusing notation, denote it again by $E$. Recall the map $\pi \colon G^k ( V, \cE ) =: \gkv \to \tUne$. By hypothesis, $\mu_E|_{\Lambda\otimes H^0(K_C\otimes E^*\otimes V^*)}$ is injective for all $\Lambda \in \Gr (k, H^0 ( V \otimes E ))$. By Proposition \ref{GkVE} (2), therefore, $\gkv$ is smooth and of dimension $\rho^k_{n, e} (V)$ in a neighbourhood of $\pi^{-1} (E)$, and $\pi^{-1} (E)$ is a smooth scheme. Moreover, by Proposition \ref{OneImpliesAll} the component of $B^k_{n, e} (V)$ containing $E$ also has dimension $\rho^k_{n, e}(V)$. As the fibres of $\pi$ are connected, being Grassmannians, $\pi$ is birational in a neighbourhood of $\pi^{-1} (E)$.

Thus the hypotheses of \cite[II, Lemma 1.1 and Corollary, p.\ 66]{ACGH} are met, with $X$ being a suitable neighbourhood of $\pi^{-1} (E)$. Hence $\bT_E \bv = d\pi ( N )$, where $N := N_{\pi^{-1} (E) / \gkv}$ is the normal bundle of the fibre $\pi^{-1} (E)$ in $\gkv$.

Next, by Proposition \ref{GkVE} (1) we have an exact sequence
\[ 0 \ \to \ T_\Lambda \Gr ( k, H^0 ( V \otimes E )) \ \to \ T_{(E, \Lambda)} \gkv \ \xrightarrow{d\pi} \ T_E U (n, e) \]
and 
\[ \Image \left( d\pi|_{(\Lambda, E)} \right) \ = \ \mu_E \left( \Lambda \otimes H^0 ( \Kc \otimes E^* \otimes V^* ) \right)^\perp \ \subseteq \ H^1 ( \End E ) . \]
Thus the total space of the normal bundle $N$ can be identified with
\[ \{ ( \Lambda, v ) \in \Gr (k, H^0 ( V \otimes E )) \times H^1 ( \End E ) : v \cup \mu_E \left( \Lambda \otimes H^0 ( \Kc \otimes E^* \otimes V^* ) \right) = 0 \} . \]
Moreover, via this identification, $d\pi$ is projection to the second factor.

This completes the proof of (4) and shows that the hypotheses of \cite[Lemma p.\ 242]{ACGH} apply to $N$ (with $I = N$, $w = k$, $A = H^0 ( V \otimes E )$ and $\phi = \mu_E$). 
 Therefore, statements (1--3) are respectively (i--iii) of \cite[Lemma p.\ 242]{ACGH}. \end{proof}

\begin{remark} If the hypotheses of Theorem \ref{main} apply, the conclusions of Proposition \ref{TangentCone} hold for general $V$ and general $E\in\bv_0$.  In particular, we can describe some tangent cones of generalised theta divisors at well behaved singular points. Suppose $g \ge r^2$ and $1 \le d \le r - 1$. Then $\frac{g}{r} \ge r$, so we may set $k_0 = d$. We write $r' := \frac{r}{\gcd(r, d)}$ and $d' := \frac{d}{\gcd(r, d)}$, and set $e_0 = g-2$. Then for any positive integer $\lambda$, the values
\[ n \ = \ \lambda r' \quad \quad \hbox{and} \quad \quad e \ = \ \lambda r' e_0 + \lambda ( r' - d') \ = \ \lambda \cdot ( r' (g-2) + r' - d' ) \]
satisfy both the hypotheses of Theorem \ref{main} and the equation $re + nd = rn(g - 1)$. (Here $e_1:=e-ne_0 = \lambda \cdot (r' - d')$. Note also that necessarily $n \ge 2$.) 
 By Theorem \ref{main}, for $1 \le k \le nk_0$ there exists a component $X_k$ of $B^k_{n, e} (V) \subset B^1_{n, e} (V)$ upon which the Petri maps $\mu_E$ are injective for general $E \in X_k$. By Proposition \ref{TangentCone} (3), for each such $E$ we have $\mult_E B^1_{n, e} (V) = h^0 ( V \otimes E )$. \end{remark}

\subsection*{Geometry of the tangent cones}

We end this section with an observation generalising \cite[Theorem 5.2]{CTiB} (see also \cite[p.\ 232]{ACGH}). Firstly, we recall from \cite[{\S} 3]{HR} that, generalising the canonical curve in $| \Kc |^*$, for any vector bundle $E$ there is a map
\begin{equation} \PP \End E \ \dashrightarrow \  |\cO_{\PP (T_C \otimes \End E)}(1)|^* \ \cong \ \PP H^0 ( \Kc \otimes \End E )^* \ = \ \PP H^1 ( \End E ) . \label{psi} \end{equation}
We write $\Delta$ for the closed sublocus $\PP E^* \times_C \PP E$ of rank one maps in $\PP \End E$. Suppose $\phi_1 , \ldots , \phi_p$ are points of $\Delta$ supported over distinct points $x_1 , \ldots , x_p$ of $C$. For $1 \le i \le p$, let $\tphi_i$ be the point corresponding to $\phi_i$ via the identification
\[ \End E|_{x_i} \ \isom \ (\End E) (x_i) |_{x_i} \ \isom \ H^0 \left( (\End E) (x_i) |_{x_i} \right) \] 
which is canonical up to nonzero scalar. We observe that for any vector bundle $W$, the element $\tphi_i$ defines a map
\begin{equation} W \otimes E|_{x_i} \ \to \ W \otimes E(x_i)|_{x_i} . \label{tphii} \end{equation}
This will be used later.

Now write $D = x_1 + \cdots + x_p$. Unwinding the definition  of the map (\ref{psi}), we see that the secant $\overline{\phi_1 , \ldots , \phi_p}$ in $\PP H^1 ( \End E )$ is the span of the images of the $\tphi_i$ by the coboundary map in the sequence
\[ 0 \ \to \ H^0 ( \End E ) \to H^0 ( (\End E) (D) ) \ \to \ H^0 ( (\End E) (D)|_D ) \ \xrightarrow{\partial} \ H^1 ( \End E ) . \]
The following is valid without any injectivity assumption on $\mu_E$.

\begin{proposition} Let $m$ be the rank of the subbundle $V_\mathrm{ggen}$ of $V$ generated by the evaluation map $E^* \otimes H^0 ( V \otimes E ) \to V$. Suppose $h^0 ( V \otimes E ) \ge pm + k$. Then $\Sec^p \Delta$ is contained in the projectivised tangent cone $\PP \bT_E B^k_{n, e} (V)$. \end{proposition}

\begin{proof} (Compare with \cite[Theorem 5.2]{CTiB}.) A tangent vector $v \in H^1 ( \End E )$ belongs to $\bT_E B^k_{n, e} (V)$ if and only if
\[ \Ker \left( \cdot \cup v \colon H^0 ( V \otimes E ) \to H^1 ( V \otimes E ) \right) \]
has dimension at least $k$. Now clearly it suffices to show that a general point of $\Sec^p \Delta$ belongs to the tangent cone. So let $v = \partial ( \lambda_1 \tphi_1 , \ldots , \lambda_p \tphi_p )$ where the $\tphi_i$ and $D$ are as above, and the $\lambda_i$ are nonzero scalars. The cup product map by $v$ factorises
\[ H^0 ( V \otimes E ) \ \xrightarrow{\ev_D} \ \bigoplus_{i=1}^p V \otimes E|_{x_i} \ \xrightarrow{( \lambda_1 \tphi_1 , \ldots , \lambda_p \tphi_p )} \ \bigoplus_{i=1}^p V \otimes E (x_i)|_{x_i} \xrightarrow{\partial'} \ H^1 ( V \otimes E ) , \]
where $\ev_D$ is the evaluation map, $\tphi_i$ is as in (\ref{tphii}) and $\partial'$ is the coboundary map of the cohomology sequence of
\[ 0 \ \to \ V \otimes E \ \to \ V \otimes E (D) \ \to \ V \otimes E(D)|_D \ \to \ 0 .\]
Now $\Image ( \ev_D )$ is contained in $\bigoplus_{i=1}^p V_\mathrm{ggen} \otimes E|_{x_i}$. In view of (\ref{tphii}), moreover
\[ \Image \left( ( \lambda_1 \tphi_1 , \ldots , \lambda_p \tphi_p ) \circ \ev_D \right) \ \subseteq \ \bigoplus_{i=1}^p V_\mathrm{ggen} \otimes \Image ( \tphi_i ) . \]
Since by hypothesis $\dim \Image(\tphi_i) = 1$, the last space has dimension $mp$. It follows that
\[ \dim \Ker ( \cdot \cup v ) \ \ge \ h^0 ( V \otimes E ) - mp \ \ge \ k . \]
The proposition follows. \end{proof}

\section{A non-empty twisted Brill--Noether locus with negative expected dimension} \label{negative}

It is well known that higher-rank Brill--Noether loci $B^k_{n,e}$ can exhibit more complicated behaviour than their rank one counterparts. Here we give an example of a nonempty twisted Brill--Noether locus with negative Brill--Noether number, where the curve $C$ and the bundle $V$ are general. Firstly, we recall some facts about \emph{maximal line subbundles} of vector bundles (see \cite{Oxb} for more general and detailed information):

 Suppose $r | (g-1)$, and set $e_0 :=(r-1) \frac{g-1}{r}$. Let $V$ be a general bundle of rank $r$ and degree zero. A computation shows that $\rho^1_{1, e_0}(V) = 0$. 
 As $V$ is general, by Theorem \ref{LTiB} (3) the locus $B^1_{1, e_0}(V)$ is of dimension zero. Furthermore, for $e < e_0$ or $k_0 > 1$ we check that $\rho^{k_0}_{1, e_0}(V) < 0$. 
 Hence, for all $L \in B^1_{1, e_0}(V)$, we have $h^0 (V \otimes L) = 1$ and $L^{-1}$ is a line subbundle of maximal degree in $V$. By \cite[Proposition 1.4 and Lemma 2.2]{Oxb}, $B^1_{1,e_0}(V)$ consists of $r^g$ points (of multiplicity $1$).

\begin{proposition} Let $r, g$ and $e_0$ be as above and let $V$ be a general bundle of rank $r\ge2$ and degree $0$. Let $n$ be an integer satisfying $r < n \le r^g$. Then the twisted Brill--Noether locus $B^n_{n, ne_0 + 1}(V)$ has negative expected dimension $n(r-n) + 1$ but contains a component of dimension at least $1$. \end{proposition}

\begin{proof} By the previous paragraph we can choose mutually nonisomorphic $L_1, \ldots , L_n \in B^1_{1, e_0} (V)$. Let 
\[ 0 \to \bigoplus_{i=1}^n L_i \to E \to \K_p \to 0 \]
be a general elementary transformation, where $\K_p$ is the skyscraper sheaf of degree $1$ supported at $p \in C$. Then $E$ is stable by \cite[Th\'eor\`eme A.5]{Mer}, and $h^0 (V \otimes E) \ge n$. The Quot scheme parametrising the elementary transformations $E$ has dimension $n$; after acting by $\operatorname{Aut}\left(\oplus_{i=1}^nL_i\right)$, we see that there is precisely one stable $E$ for any given $p$. Thus $\dim B^n_{ne_0 + 1} (V) \ge 1$. On the other hand, we compute easily that
\[ \rho^n_{n, ne_0 + 1}(V) \ = \ n(r-n) + 1 . \]
Since by hypothesis $n > r \ge 2$, this is negative. The result follows. \end{proof}

In exactly the same way, one can prove

\begin{proposition} Let $C$ be a curve and $V$ a bundle of rank $r \ge 2$ and degree $d$ over $C$. Suppose $k_0 \ge 1$ and $e_0$ are integers satisfying $\rho^{k_0}_{1, e_0}(V) = 0$. Assume that
\[ \# \left( B^{k_0}_{1, e_0}(V) \right) \ > \ k_0 r . \]
Then for $rk_0 < n \le \# \left( B^{k_0}_{1, e_0}(V) \right)$, the twisted Brill--Noether locus $B^{nk_0}_{n, ne_0 + 1}(V)$ is non-empty and has negative expected dimension. \label{NegativeExpDim} \end{proposition}



\noindent This example shows that even for general stable $V$, the twisted Brill--Noether loci can exhibit pathologies.

\end{document}